\documentclass[12pt,a4paper]{article}

\usepackage{amsmath,amsthm,amssymb,amscd,a4wide}

\usepackage{dsfont}
\usepackage{mathrsfs}
\usepackage{setspace}
\usepackage{hyperref}

\usepackage{xcolor}

\newcommand{\ud}{\mathrm{d}}

\numberwithin{equation}{section}

\newtheorem{theorem}{Theorem}[section]
\newtheorem{lemma}[theorem]{Lemma}
\newtheorem{prop}[theorem]{Proposition}

\newtheorem{remark}[theorem]{Remark}

\theoremstyle{definition}

\newtheorem{rem}[theorem]{Remark}

\numberwithin{equation}{section}

\begin{document}

\thispagestyle{empty}

\vspace*{1cm}

\begin{center}

{\LARGE\bf A lower bound on the spectral gap of one-dimensional Schrödinger operators} \\

\vspace*{2cm}

{\large Joachim Kerner \footnote{E-mail address: {\tt Joachim.Kerner@fernuni-hagen.de}}}%

\vspace*{5mm}

Department of Mathematics and Computer Science\\
FernUniversit\"{a}t in Hagen\\
58084 Hagen\\
Germany\\

\end{center}

\vfill

\begin{abstract} In this note we provide an explicit lower bound on the spectral gap of one-dimensional Schrödinger operators with non-negative bounded potentials and subject to Neumann boundary conditions. 
\end{abstract}

\newpage

%\input{intro}
%%%
\section{Introduction}
In this paper, our aim is to derive a lower bound on the spectral gap for (deterministic) Schrödinger operators defined on an interval of length $L > 0$, subject to Neumann boundary conditions, and with non-negative potentials in the class $L^{\infty}(\mathbb{R})$. Our bound will be valid for all values of $L > 0$ and will depend on the underlying potential in an explicit way. For convenience, we also provide some additional results which are not necessary to derive the lower bound but which might be of some interest nevertheless.

%As recognized in \cite{KT20}, one gains a certain flexibility when studying the asymptotics of the gap rather than the gap on an interval of fixed length. In particular, one is able to allow for more general potentials which do not satisfy, e.g., a convexity assumption. Assumptions of this type, on the other hand, are often invoked when studying the influence of generic potentials on the spectral gap on fixed intervals \cite{AB,Abramovich,Lavine}. 

The paper also forms a natural continuation of the investigations started recently in \cite{KT20}. There, it was shown that the spectral gap of Schrödinger operators (subject to Dirichlet boundary conditions) closes, in the limit of $L \rightarrow \infty$ and for a large class of potentials, strictly faster than the spectral gap of the free (Dirichlet) Laplacian. This holds, in particular, for non-zero bounded potentials $v:\mathbb{R} \rightarrow \mathbb{R}_+$ of compact support. More explicitly, for a symmetric step-potential and all $L > 0$ large enough, lower and upper bounds on the spectral gap were derived showing that the gap closes like $\sim L^{-3}$ and it was conjectured that the gap cannot close faster than $\sim L^{-3}$ for bounded potentials of compact support. However, no lower bound on the spectral gap was derived in general and this is what motivated the investigations in this paper. Since we employ a version of Harnack's inequality, our lower bound is (for non-zero potentials) at least exponentially small in the interval length and therefore does not lead to a proof of a conjecture as formulated in \cite{KT20}. Nevertheless, it provides a first step and the bound obtained is valid in greater generality; in particular, it is not restricted to the limit of large intervals.\\
                 Finally, let us refer to \cite{KirschSimongapOOO} where lower bounds on the spectral gaps between consecutive eigenvalues for one-dimensional Schrödinger operators have been studied in a different setting.
                                                                        
\section{The model and results}
On the interval $I=(-L/2,+L/2)$ we consider a Schrödinger operator of the form
\begin{equation*}
h_L=-\frac{\ud^2}{\ud x^2}+v
\end{equation*}
with real, non-negative potentials $v \in L^{\infty}(\mathbb{R})$ and Neumann boundary conditions at the endpoints $x=\pm L/2$. Standard operator theory tells us that $h_L$ is self-adjoint with purely discrete spectrum. We denote its eigenvalues as $\lambda_0(L)\leq \lambda_1(L) \leq ...$; the normalized ground-state eigenfunction shall be denoted as $\varphi^0_L \in L^2(I)$. We also recall that $\varphi^0_L$ is a positive function.

Our main object of interest is the spectral gap 
\begin{equation*}
\Gamma_v(L):=\lambda_1(L)-\lambda_0(L)\ .
\end{equation*}
We remark that, since the ground state is non-degenerate \cite{LiebLoss}, the spectral gap is strictly larger than zero for every value $L > 0$. Also, $\Gamma_v(L)$ converges to zero as $L \rightarrow \infty$ for potentials that decay sufficiently fast at infinity (for example, as in [Theorem~2.1,\cite{KT20}], it suffices to assume $|v(x)|\leq \frac{C}{|x|^2}$ for some $C >0$ and almost all $x \in \mathbb{R}$).

%\begin{prop} For the ground state $\varphi^0_L$ we have 
%	%
%	\begin{equation*}
%	c_1 \leq \varphi^0_L(x) \leq c_2 \ , \quad \text{for a.e.}\ x \in I\ .
%	\end{equation*}
	%
%\end{prop}
%
In a first result we provide an upper bound on the supremum (or maximum) of the ground-state eigenfunction for potentials that decay at least quadratically at infinity. We note that such a class of short-range potentials was considered in \cite{KT20}. 
\begin{prop}\label{UpperBoundSup} Let $v \in L^{\infty}(\mathbb{R})$ be non-negative and such that 
	\begin{equation*}
	|v(x)| \leq \frac{C}{|x|^{2}}
	\end{equation*}
	for some constant $C > 0$ and almost all $x \in \mathbb{R}$. Then
	\begin{equation}\label{XXX}\begin{split}
	\sup_{x \in I}|\varphi^0_L(x)|\leq  \frac{1+\sqrt{4\pi^2+16C}}{\sqrt{L}}
	\end{split}
	\end{equation}
	holds for all $L > 0$.
\end{prop}
\begin{proof} In a first step we realize that, for every $L > 0$, there exists $x^L_0 \in I$ such that $\varphi^0_L(x^L_0)=\frac{1}{\sqrt{L}}$. Furthermore,
	\begin{equation*}
	\varphi^0_L(x)=\int_{x^L_0}^x(\varphi^0_L(t))^{\prime}\ \ud t+\frac{1}{\sqrt{L}}\ ,
	\end{equation*}
	which yields, for $x \in I$ and all $L > 0$,
	\begin{equation}\label{CCC}\begin{split}
	|\varphi^0_L(x)| &\leq \sqrt{\lambda_0(L)} \cdot \sqrt{L}+\frac{1}{\sqrt{L}} \ ,
	\end{split}
	\end{equation}
making use of the minmax-principle \cite{schmudgen2012unbounded}. Now, for the potentials considered in Proposition~\ref{UpperBoundSup}, the minmax-principle also gives the bound
\begin{equation}\label{explicBound}
\lambda_0(L) \leq \frac{4\pi^2+16C}{L^2}\ ,
\end{equation}
valid for all $L > 0$. To see this, one picks as a trial function the (normalized) ground state of the Laplacian on the interval $(+L/4,+L/2)$ subject to Dirichlet boundary conditions at $x=L/4$ and Neumann boundary conditions at $x=L/2$. Eq.~\eqref{explicBound} then follows from the bound 
\begin{equation*}
\lambda_0(L) \leq \frac{\pi^2}{(L/2)^2}+\|v\|_{L^{\infty}((+L/4,+L/2))} 
\end{equation*}
which itself follows from the Rayleigh quotient. The statement then follows readily combining~\eqref{explicBound} and~\eqref{CCC}.
\end{proof}                                                              
In a next step we derive a lower bound to the infimum of the ground-state eigenfunction as well as a Harnack-type inequality for a larger class of potentials.
\begin{lemma}\label{LB} Assume $v \in L^\infty(\mathbb{R})$ and $v \geq 0$. Then 
	\begin{equation}\label{EQINF}
	\inf_{x \in I}|\varphi^0_L(x)| \geq \frac{\mathrm{e}^{-4L\|v\|_{L^1(I)}}}{\sqrt{L}}
	\end{equation}
	holds for all $L > 0$. Furthermore, for all $L > 0$, one has 
	\begin{equation}\label{HarnackINEQ}
		\inf_{x \in I}|\varphi^0_L(x)| \geq \mathrm{e}^{-4L\|v\|_{L^1(I)}} \cdot	\sup_{x \in I}|\varphi^0_L(x)|\ .
	\end{equation}
	.
\end{lemma}
\begin{proof} We follow the strategy outlined in \cite{MR2110110} and, in particular, the proof of [Theorem~1.2,\cite{MR2110110}] . [Eq.~(2.3),\cite{MR2110110}] implies, for all $x \in I$,
	\begin{equation*}
	\int_I\left(q(x)-|q(x)|\right)\ \ud x \leq \frac{(\varphi^0_L)^{\prime}(x)}{\varphi^0_L(x)}\leq \|q\|_{L^1(I)}\ ,
	\end{equation*}
	where $q(x):=\lambda_0(L)-v(x)$. This immediately gives
	\begin{equation*}
	\left|\frac{(\varphi^0_L)^{\prime}(x)}{\varphi^0_L(x)} \right| \leq 2\|q\|_{L^1(I)}\ .
	\end{equation*}
	Note that $\|q\|_{L^1(I)} \leq 2\|v\|_{L^1(I)}$ since
	\begin{equation*}
	\lambda_0(L) \leq \frac{\|v\|_{L^1(I)}}{L}
	\end{equation*}
	by the minmax-principle (using $1/\sqrt{L}$ as a test function in the Rayleigh quotient). Furthermore, this shows that the constant $C_1$ in the last steps of the proof of [Theorem~1.2,\cite{MR2110110}] can be chosen to be $4\|v\|_{L^1(I)}$ showing that
	\begin{equation}\label{HarnackIEQ}
	\frac{\varphi^0_L(y)}{\varphi^0_L(x)} \leq \mathrm{e}^{4L\|v\|_{L^1(I)}}\ , \quad x,y \in I\ .
	\end{equation}
Eq.~\eqref{HarnackIEQ} then immediately implies~\eqref{HarnackINEQ}. Finally, noting that
	\begin{equation*}
	\sup_{x \in I}|\varphi^0_L(x)| \geq \frac{1}{\sqrt{L}}
	\end{equation*}
	one arrives at~\eqref{EQINF}. 
\end{proof}
\begin{rem} Eq.\eqref{EQINF} in Lemma~\ref{LB} is sharp: choosing $v \equiv 0$ yields the lower bound $\frac{1}{\sqrt{L}}$. On the other hand, for the zero-potential, the ground state is given by $\varphi^L_0(x)=\frac{1}{\sqrt{L}}$.
\end{rem}
Using [Theorem~1.4,\cite{KirschGap}] in combination with Lemma~\ref{LB} then yields the main result of the paper.
\begin{theorem}[Lower bound spectral gap] \label{MainResult}
	Assume $v \in L^\infty(\mathbb{R})$ and $v \geq 0$.  Then,
	\begin{equation*}\begin{split}
	\Gamma_v(L) \geq \mathrm{e}^{-8L\|v\|_{L^1(I)}} \cdot \frac{\pi^2}{L^2} 
	\end{split}
	\end{equation*}
	holds for all $L > 0$. 
\end{theorem}
\begin{proof} The idea is to compare the spectral gap of the operator $h_L$ with the spectral gap of the Neumann Laplacian (i.e., setting $v \equiv 0$) which is given by $\pi^2/L^2$. Such a comparison result has been provided in [Theorem~1.4,\cite{KirschGap}]: more explicitly, one has
	\begin{equation*}\begin{split}
	\Gamma_v(L) &\geq  \left(\frac{\inf_{x \in I} |\varphi^0_L(x)|}{\sup_{x \in I} |\varphi^0_L(x)|}\right)^2 \cdot \frac{\pi^2}{L^2} \ ,
	\end{split}
	\end{equation*}
	taking into account that the normalized ground-state eigenfunction for the Neumann Laplacian is the constant function $1/\sqrt{L}$. The result then readily follows with Lemma~\ref{LB}.
\end{proof}
\begin{remark} Theorem~\ref{MainResult} establishes a lower bound which is (for non-zero potentials) at least exponentially small in the interval length. Hence, this bound is still far away from a lower bound as established, for example, in [Proposition~2.9,\cite{KT20}] for a symmetric step-potential. In this case, the lower bound reads $\alpha L^{-3}$ for some constant $\alpha > 0$ and all $L > 0$ large enough. However, due to the fact that a Harnack-type inequality has been used in the proof of Theorem~\ref{MainResult}, an exponential factor seems expectable.
	
	Also, for the zero-potential $v \equiv 0$, the lower bound in Theorem~\ref{MainResult} is sharp.
\end{remark}
%
%\begin{remark} Using the estimate \eqref{explicBound} on the ground-state eigenvalue in the proof of Lemma~\ref{LB} would allow one to improve slighlty on some constants in Lemma~\ref{LB} and Theorem~\ref{MainResult}. 
%\end{remark}
%
%\begin{remark} Theorem~\ref{MainResult} can be generalized in a straightforward manner to a larger class of potentials. Indeed, for any non-negative potential $v \in L^{\infty}(\mathbb{R})$, one has the bound
	%
%	\begin{equation*}\begin{split}
%	\Gamma_v(L) \geq \mathrm{e}^{-8L\|v\|_{L^1(I)}} \cdot \frac{\pi^2}{L^2} 
%	\end{split}
%	\end{equation*}
	%
%	which holds for all $L > 0$.
%\end{remark}

%\input{2sec}
%\input{3sec}

\vspace*{0.5cm}

\subsection*{Acknowledgement}{}JK would like to thank M.~Täufer for helpful discussions.

\vspace*{0.5cm}

{\small
\bibliographystyle{amsalpha}
\bibliography{Literature}}

\end{document}